\documentclass{amsart}
\usepackage{amsmath,amssymb,amsthm,amsfonts,amscd,mathrsfs}
\usepackage[all]{xy}
\usepackage[dvips]{graphicx}

\newcommand{\be}{\begin{enumerate}}
\newcommand{\ee}{\end{enumerate}}

\newcommand{\Z}{\mathbb{Z}}
\newcommand{\C}{\mathbb{C}}

\newcommand{\co}{\colon\thinspace}

\newcommand{\TC}{{\sf TC}}

\newcommand{\rank}{\mathrm{rank}}
\newcommand{\cd}{\mathrm{cd}}
\newcommand{\secat}{{\sf secat}}
\newcommand{\cat}{{\sf cat}}
\newcommand{\ev}{\mathrm{ev}}

\newtheorem{thm}{Theorem}[section]
\newtheorem{lemma}[thm]{Lemma}
\newtheorem{defn}[thm]{Definition}
\newtheorem{exam}[thm]{Example}
\newtheorem{prop}[thm]{Proposition}

\newtheorem{cor}[thm]{Corollary}

\newtheorem{rem}[thm]{Remark}

\begin{document}

\title[Topological complexity and symmetry]{Topological complexity, fibrations and symmetry}

\author{Mark Grant}
\address{School of Mathematics, The University of Edinburgh, King's Buildings, Edinburgh, EH9 3JZ, UK}
\email{mark.grant@ed.ac.uk}

\keywords{Topological complexity, fibrations, transformation groups}
\subjclass[2000]{55M99 (Primary); 55M30, 57S15, 68T40 (Secondary)}

\maketitle
\begin{abstract} We show how locally smooth actions of compact Lie groups on a manifold $X$ can be used to obtain new upper bounds for the topological complexity $\TC(X)$, in the sense of Farber. We also obtain new bounds for the topological complexity of finitely generated torsion-free nilpotent groups.
\end{abstract}

\section{Introduction}

The topological approach to the motion planning problem in Robotics was initiated by Farber in \cite{Far1}, \cite{Far2}. To each space $X$ is associated a natural number $\TC(X)$, the {\em topological complexity}, which is an invariant of homotopy type and quantifies the complexity of the task of navigation in $X$. Knowledge of $\TC(X)$ is of practical use when designing optimal motion planning algorithms for mechanical systems whose configuration space is of the homotopy type of $X$. In topological terms, $\TC(X)$ is the sectional category (or Schwarz genus) of the free path fibration
\[
\pi\co \mathscr{P}X\to X\times X,\qquad\pi(\gamma)=\big(\gamma(0),\gamma(1)\big),
\]
where $\mathscr{P}X=\{\gamma\co I\to X\}$ denotes the space of paths in $X$ with the compact-open topology.

The invariant $\TC(X)$ is a close relative of the Lusternik-Schnirelmann category $\cat(X)$ (the minimum cardinality amongst covers of $X$ by open sets whose inclusions are null-homotopic), although the two are independent. In fact
\[
\cat(X)\leq \TC(X)\leq\cat(X\times X)
\]
for any space $X$, and either inequality can be an equality (one does not have to look further than the orientable surfaces for examples, see \cite{Far1} Theorem 9). There is strong evidence that $\TC(X)$ captures finer information about the homotopy type of $X$ than does $\cat(X)$, and is therefore harder to compute. For instance, computing the topological complexity of real projective spaces is equivalent to solving the immersion problem for these manifolds \cite{FTY}, whereas $\cat(P^n)=n+1$. Another striking example is given by aspherical spaces, whose category equals the cohomological dimension of their fundamental group (see \cite{EG57}, \cite{Sta}). However no such algebraic description of the topological complexity is known, even conjecturally (see the introduction to \cite{CP} for a discussion of this problem and a survey of known results).

In order to compute $\TC(X)$, one looks for approximating invariants. Lower bounds come from the cohomology algebra, in terms of the so-called zero-divisors cup-length \cite{Far1}. These can be sharpened using cohomology operations, both stable \cite{FG07} and unstable \cite{G}. A general upper bound can be given in terms of covering dimension. Namely, if $X$ is a paracompact space, then $\TC(X)\leq 2\dim(X)+1$. If $X$ is simply-connected, this can be strengthened to $\TC(X)\leq\dim(X)+1$.

In this paper we investigate the relationship of topological complexity with compact group actions. For $G$ a compact Lie group acting on a manifold $X$, we denote by $G_x = \lbrace g\in G \mid g\cdot x=x\rbrace$ the stabiliser subgroup of a point $x\in X$, and by $F(H,X)=\lbrace x\in X\mid h\cdot x = x\mbox{ for all }h\in H\rbrace$ the fixed point set of a subgroup $H\subseteq G$. A principal orbit is a special type of orbit of maximal dimension. In section 5 below, we obtain the following result.
\begin{thm}[Theorem \ref{path}]
Let $X$ be a closed, connected smooth manifold. Let $G$ be a compact Lie group acting locally smoothly on $X$, with principal orbit $P$. Suppose that for all pairs of points $x,y\in X$, either of the following conditions holds:
\be
\item The fixed-point set $F(G_x\cap G_y,X)$ is path-connected;
\item $x,y\in F(G,X)$ are fixed points.
\ee
Then
\[
\TC(X)\leq 2\dim(X)-\dim(P)+1.
\]
\end{thm}
We apply this result to obtain upper bounds on the topological complexity of various spaces of interest, such as manifolds with free or semi-free group actions, homology spheres and mapping tori of periodic diffeomorphisms.

The above result on group actions is obtained by applying the following result to the orbit map $p\co X\times X\to (X\times X)/G$ of the diagonal action.
\begin{thm}[Theorem \ref{main}]
Let $X$ be a normal ENR, and let $q\co X\times X\to Y$ be a closed map with $Y$ paracompact. Suppose further that $\TC_X(q^{-1}(y))\leq n$ for each $y\in Y$. Then
\[
\TC(X)\leq(\dim Y+1)\cdot n.
\]
\end{thm}
The definition of the subspace topological complexity $\TC_X(A)$ where $A\subseteq X\times X$ can be found in section 2 below. As another application of Theorem \ref{main} we give a quick proof of (part of) the Theorem of Farber, Tabachnikov and Yuzvinsky \cite{FTY} on immersion dimension (see Corollary \ref{axial}).

  We also discuss the topological complexity of fibrations. The main result of section 3 is the following strengthening of Lemma 7 in \cite{FG07}.
\begin{thm}[Theorem \ref{fibration}]
Let $\xymatrix{F \ar[r] & E\times E \ar[r]^-q & Y}$ be a fibration with $Y$ path-connected. Then
\[
\TC(E)\leq \cat(Y)\cdot\TC_E(F).
\]
\end{thm}
We apply this result to bound the topological complexity of aspherical spaces with nilpotent fundamental group.
\begin{cor}[Corollary \ref{nil}]
Let $\Gamma$ be a finitely generated torsion-free nilpotent group with centre $\mathcal{Z}\leq \Gamma$. Then
\[
\TC(\Gamma)\leq 2\cdot\rank(\Gamma)-\rank(\mathcal{Z})+1.
\]
\end{cor}

The results and proofs presented here are very much influenced by the corresponding results for Lusternik-Schnirelmann category, obtained by Oprea and Walsh \cite{OW}. The author would like to thank John Oprea and Greg Lupton for stimulating discussions, and Michael Farber, Peter Landweber and the anonymous referee for their comments on earlier drafts of the paper.

\section{Topological Complexity}

In this section we collect several definitions and results pertaining to numerical invariants, beginning with the Lusternik-Schnirelmann category of a space pair $(X,B)$.
\begin{defn}[\cite{CLOT}]\label{subcat}{\em
Let $X$ be a path-connected space with subspace $B\subseteq X$. The {\em subspace category} of $B$ in $X$, denoted $\cat_X(B)$, is the smallest integer $k$ for which $B$ admits an open cover $B=U_1\cup \dots\cup U_k$ with each composition of inclusions $U_j\hookrightarrow B\hookrightarrow X$ null-homotopic. If no such integer exists we set $\cat_X(B)=\infty$.\\

Note that $\cat_X(\emptyset)=0$ and $\cat_X(X)=\cat(X)$, the usual Lusternik-Schnirelmann category of $X$. Note also that $\cat_X(B)=1$ if and only if the inclusion $B\hookrightarrow X$ is null-homotopic.
}\end{defn}

\begin{rem}{\em Many authors prefer to normalise this definition (so that for instance $\cat_X(B)=0$ when the inclusion $B\hookrightarrow X$ is null-homotopic). Here we choose not to do so, mainly to jibe with the existing literature on topological complexity. }
\end{rem}

\begin{defn}[\cite{Sch66},\cite{CLOT}]\label{secat}{\em
Let $p\co E\rightarrow X$ be a (Hurewicz) fibration. The {\em sectional category} of $p$, denoted $\secat(p)$ (also called the {\em Schwarz genus} of $p$) is the smallest integer $k$ for which the base $X$ admits a cover $X=U_1\cup \dots\cup U_k$ by open sets, each of which admits a continuous local section $s_i\co U_i\rightarrow E$ of $p$ (that is, $s_i$ is continuous and $p\circ s_i$ is the inclusion $U_i\hookrightarrow X$). If no such integer exists we set $\secat(p)=\infty$.
}\end{defn}

We remark that Definition \ref{secat} generalises Definition \ref{subcat}, in the following sense. For a path-connected space $X$, let $PX=\{\gamma\co I=[0,1]\rightarrow X\mid \gamma(0)=x_0\}$ denote the space of all paths in $X$ emanating from a fixed base-point $x_0\in X$, endowed with the compact-open topology. It is well known that the end-point map
\[
p\co PX\rightarrow X,\qquad p(\gamma)=\gamma(1)
\]
is a fibration. Given any subspace $B\subseteq X$ we have $\cat_X(B)=\secat(p|_B)$, where $p|_B\co p^{-1}B\rightarrow B$ denotes the restriction to paths terminating in $B$.

For any space $X$, let $\mathscr{P}X=\{\gamma\co I\rightarrow X\}$ denote the space of free paths in $X$, endowed with the compact-open topology. It is well known (see Spanier \cite{Spa}) that the end-point map
\[
\pi\co \mathscr{P}X\rightarrow X\times X,\qquad\pi(\gamma)=\big(\gamma(0),\gamma(1)\big)
\]
is a fibration.

\begin{defn}[\cite{Far1},\cite{Far2},\cite{Far3}]{\em
Let $X$ be a path-connected space, with $A\subseteq X\times X$ a subspace of the product. The {\em subspace topological complexity} of $A$ in $X$ is defined to be
\[
\TC_X(A)=\secat(\pi|_A),
\]
where $\pi|_A\co \pi^{-1}A\rightarrow A$ denotes the restriction to paths whose pair of initial and terminal points lies in $A$.
}\end{defn}

Note that $\TC_X(\emptyset)=0$ and $\TC_X(X\times X)=\TC(X)$, the usual topological complexity of $X$.

\begin{lemma}[\cite{Far3}] For a nonempty subspace $A\subseteq X\times X$, the following are equivalent:
\be
\item $\TC_X(A)=1$;
\item The projections $\xymatrix{ X & A \ar[l]_-{p_1} \ar[r]^-{p_2} & X}$ are homotopic;
\item The inclusion $A\hookrightarrow X\times X$ is homotopic to a map with values in the diagonal $\triangle(X)=\{(x,x)\}\subseteq X\times X$.
\ee
\end{lemma}
The proof is straightforward and is omitted. More generally, $\TC_X(A)$ is the smallest integer $k$ for which $A$ admits an open cover $A=U_1\cup \dots\cup U_k$ with each composition of inclusions $U_j\hookrightarrow A\hookrightarrow X\times X$ homotopic to a map with values in $\triangle(X)$.

\begin{lemma}\label{rels} Let $X$ ba a path-connected space with subspace $B\subseteq X$. Then
\[
\cat_X(B)\leq\TC_X(B\times B)\leq\cat_{X\times X}(B\times B).
\]
\end{lemma}
\begin{proof}
The second inequality is obvious (since if $U\subseteq B\times B$ is null-homotopic in $X\times X$, then it is homotopic into the diagonal). To obtain the first inequality, suppose $\TC_X(B\times B)=k$. Then we have an open cover $B\times B=U_1\cup\cdots \cup U_k$ and local sections $s_j\co U_j\rightarrow \mathscr{P}X$ of $\pi$ for $j=1,\ldots , k$. Choose a basepoint $b_0\in B$ and consider the map $i\co B\rightarrow B\times B$, $i(b)=(b,b_0)$. The sets $V_j=i^{-1}U_j$ cover $B$, are open, and admit contractions in $X$ defined by setting $\sigma_j\co V_j\times I\rightarrow X$, $\sigma_j(b,t)=s_j(b,b_0)(t)$.
\end{proof}
We record here the relation with covering dimension \cite{Sch66}: if $A\subseteq X\times X$ is paracompact, then
\begin{equation}\label{dim}
\TC_X(A)\leq \dim(A)+1.
\end{equation}
Many other properties of the subspace topological complexity are discussed in \cite{Far3}, Chapter 4. Here we will need the following lemma.

\begin{lemma}\label{open} Let $X$ be a normal ENR. If $A\subseteq X\times X$ is closed and $TC_X(A)\leq n$, then there exist open sets $W_1,\ldots ,W_n$ in $X\times X$ such that $A\subseteq \bigcup_i W_i$ and $\pi\co \mathscr{P}X\rightarrow X\times X$ admits a local section on each $W_i$.
\end{lemma}

\begin{proof}
We may cover $A$ by sets $U_1,\ldots ,U_n$ open in $A$ such that $\pi$ admits a local section on each $U_i$. Using normality of $A$ we obtain another cover $V_i$ such that $V_i\subseteq \overline{V_i}\subseteq U_i$ for all $i=1,\ldots ,n$. Note that $\overline{V_i}$ is closed in $X\times X$.

The projections $p_1, p_2\co \overline{V_i}\rightarrow X$ are homotopic, and since $X\times X$ is a normal ENR there exist open sets $W_i\supseteq \overline{V_i}$ in $X\times X$ such that $p_1,p_2\co W_i\rightarrow X$ are homotopic (by the conclusion of Exercise IV.8.2 of \cite{Do}). The $W_i$ cover $A$, and admit local sections of $\pi$, completing the proof.
\end{proof}

\section{Fibrations}

It is well known that the Lusternik-Schnirelmann category behaves sub-additively with respect to fibrations. That is, if $\xymatrix{F \ar[r] & E \ar[r]^-p & B}$ is a fibration, then
\begin{equation}\label{cat}
\cat(E)\leq \cat(B)\cdot\cat_E(F)\leq \cat(B)\cdot\cat(F).
\end{equation}
Here we investigate analogous inequalities for topological complexity. We remark that, for a fibration $p$ as above, the question of whether the  inequality
\begin{equation}
\TC(E)\leq \TC(B)\cdot\TC(F)
\end{equation}
holds in general remains open.

In Lemma 7 of \cite{FG07} it was noted that, if $\xymatrix{F \ar[r] & E \ar[r]^-p & B}$ is a fibration, then
\begin{equation}\label{true}
\TC(E)\leq \cat(B\times B)\cdot\TC(F).
\end{equation}
This fact is obtained using the lifting properties of the two-fold product fibration $\xymatrix{F\times F \ar[r] & E\times E \ar[r]^-{p\times p} & B\times B}$. Here we strengthen inequality (\ref{true}) in two ways: by considering fibrations $q\co E\times E\rightarrow Y$ which aren't necessarily products; and by replacing $\TC(F)$ with the potentially smaller number $\TC_E(F\times F)$ in the product case.

\begin{thm}\label{fibration}
Let $\xymatrix{F \ar[r] & E\times E \ar[r]^-q & Y}$ be a fibration with $Y$ path-connected. Then
\begin{equation}\label{fibrineq}
\TC(E)\leq \cat(Y)\cdot\TC_E(F).
\end{equation}
In particular, applied to the product fibration $\xymatrix{F\times F \ar[r] & E\times E \ar[r]^-{p\times p} & B\times B}$ this gives
\begin{equation}\label{prodineq}
\TC(E)\leq \cat(B\times B)\cdot\TC_E(F\times F).
\end{equation}
\end{thm}
\begin{proof} Choose a base-point $b\in Y$; the relative topological complexity in the statement refers to the fibre $F=F_b\subseteq E\times E$. Let $\cat(Y)=k$, and assume $Y$ is covered by open sets $V_1,\ldots , V_k$ each having a contraction into $\{b\}$. The sets $G_j=q^{-1}(V_j)$ for $j=1,\ldots , k$ form an open cover of $E\times E$. For each index $j$ the fibration property gives a homotopy $H^j_t\co G_j\rightarrow E\times E$ with $H^j_0$ the inclusion and $H^j_1(G_j)\subseteq F$. Now let $\TC_E(F)=\ell$, and assume $F$ covered by open sets $W_1,\ldots , W_\ell$ on each of which $\pi$ admits a local section $s_i\co W_i\rightarrow \mathscr{P}E$. The sets $U_{ij}:=(H^j_1)^{-1}(W_i)$ for $i=1,\ldots ,\ell$, $j=1,\ldots ,k$ form an open cover of $E\times E$, and we claim that there is a local section $\sigma_{ij}$ of $\pi$ on each of them. Informally, given a point $(x,y)\in U_{ij}$ the path $\sigma_{ij}(x,y)\co I\rightarrow E$ from $x$ to $y$ is the concatenation of three paths: first, the projection of $H^j_t(x,y)$ on to the first coordinate; second, the path $s_i(H^j_1(x,y))$; finally, the projection of $H^j_{1-t}(x,y)$ onto the second coordinate. More explicitly,
\[
\sigma_{ij}(x,y)(t) = \left\{\begin{array}{ll} p_1\circ H^j_{3t}(x,y) & 0\leq t\leq 1/3; \\
                                                s_i(H^j_1(x,y))(3t-1) & 1/3 < t \leq 2/3; \\
                                                p_2\circ H^j_{3-3t}(x,y) & 2/3 < t \leq 1.  \end{array}\right.
\]
\end{proof}
\begin{rem} {\em The above proof generalises immediately to show that
\[
\TC(E)\leq \cat(q)\cdot\TC_E(F),
\]
where $\cat(q)$ denotes the category of the map $q\co E\times E\rightarrow Y$ (recall that the {\em category of a map} $f\co X\rightarrow Y$ is the smallest integer $k$ such that $X$ admits an open cover $X=U_1\cup \cdots \cup U_k$ with each restriction $f|_{U_i}\co U_i\rightarrow Y$ null-homotopic). However we will not require this level of generality here.}
\end{rem}

Can inequality (\ref{prodineq}) ever yield better upper bounds for $\TC(E)$ than inequality (\ref{true})?
Suppose $\xymatrix{F \ar[r]^-i & E \ar[r]^-p & B}$ is a fibration with non-contractible fibre, such that the fibre inclusion $i\co F\rightarrow E$ is null-homotopic. Then the map $i\times i\co F\times F \rightarrow E\times E$ is also null-homotopic, and we have $\cat_E(F)=\TC_E(F\times F)=\cat_{E\times E}(F\times F)=1$ from Lemma \ref{rels}, while $\TC(F)>1$. Hence inequality (\ref{prodineq}) gives
\begin{equation}\label{useless}
\TC(E)\leq \cat(B\times B),
\end{equation}
which we could not have concluded from inequality (\ref{true}) alone. However, inequality (\ref{useless}) can be obtained from the corresponding inequality (\ref{cat}) for category applied to the fibration $p\times p\co E\times E\rightarrow B\times B$, since
\[
\TC(E)\leq \cat(E\times E) \leq\cat(B\times B)\cdot \cat_{E\times E}(F\times F) = \cat(B\times B),
\]
where the first inequality is completely general (see \cite[Theorem 5]{Far2}).

\begin{exam}{\em If $p\co \widetilde{X}\rightarrow X$ is a covering space with $\widetilde{X}$ path-connected, then
\[
\TC(\widetilde{X})\leq\cat(X\times X).
\]
}\end{exam}
\begin{exam}[{\cite[Example 4.5]{OW}}]{\em The complex Stiefel fibrations
\[
\xymatrix{U(k) \ar[r] & V_{k,n}(\C) \ar[r]^-\rho & G_{k,n}(\C)}
\]
have null-homotopic fibre inclusions for $2k\leq n$, where $V_{k,n}(\C)$ denotes the space of $k$-frames in $\C^n$. Hence when $2k\leq n$ we may conclude that
\[
\TC(V_{k,n}(\C))\leq \cat(G_{k,n}(\C)\times G_{k,n}(\C))=2k(n-k)+1.
\]
}
\end{exam}
\begin{rem}{\em
It may be that there exist fibrations $\xymatrix{F \ar[r]^i & E \ar[r]^p & B}$ for which
\[
1 < \TC_E(F\times F) < \TC(F).
\]
We do not currently know of any examples (although we note that the latter inequality holds whenever $\TC(E)<\TC(F)$).
}\end{rem}
We now turn to cases where the more general (\ref{fibrineq}) may be applied.

\begin{exam}{\em Let $G$ be a connected topological group. The map $\mu\co G\times G\to G$ given by $(g,h)\mapsto gh^{-1}$ is a principal $G$-bundle. (To see this, consider the free right $G$-action of $G$ on $G\times G$ given by $(g,h)\cdot g_1 = (gg_1,hg_1)$, and note that $\mu$ can be identified with the orbit map $G\times G\to (G\times G)/G$ of this action.) Since the fibre over the identity is the diagonal $\triangle G\subseteq G\times G$, Theorem \ref{fibration} gives
 \[
 \TC(G)\leq \cat(G).
 \]
 As $\cat(X)\leq\TC(X)$ for any space (\cite[Theorem 5]{Far3}), this recovers the fact that $\TC(G)=\cat(G)$, first proved as Lemma 8.2 of \cite{Far2}.}
\end{exam}

In \cite{FarS}, Farber posed the problem of computing $\TC(\Gamma):=\TC(K(\Gamma,1))$, where $\Gamma$ is a torsion-free discrete group, in algebraic terms. The following result may be useful in this regard.

\begin{prop} Let $\Gamma$ be a torsion-free discrete group, and let $\mathcal{Z}\leq \Gamma$ be its centre. Identify $\mathcal{Z}$ with its image under the diagonal embedding $d\co \Gamma\to \Gamma\times\Gamma$ (which is also a normal subgroup). Then
\[
\TC(\Gamma)\leq \cat\big((\Gamma\times \Gamma)/\mathcal{Z}\big).
\]
\end{prop}
\begin{proof}
Letting $H:=(\Gamma\times \Gamma)/\mathcal{Z}$ denote the quotient group, we have a group extension
\begin{equation}\label{extension}
\xymatrix{
\{1\} \to\mathcal{Z} \ar[r]^-{d|_{\mathcal{Z}}} & \Gamma\times\Gamma \ar[r] & H \to \{1\}}
\end{equation}
which leads to a fibration of the corresponding Eilenberg-Mac Lane spaces
\[
\xymatrix{
K(\mathcal{Z},1) \ar[r]^-{K(d|_{\mathcal{Z}},1)} & K(\Gamma\times\Gamma,1) \ar[r] & K(H,1).}
\]
Since the homomorphism $d|_\mathcal{Z}$ factors through $d$, the map $K(d|_\mathcal{Z},1)$ factors through $K(d,1)\co K(\Gamma,1)\to K(\Gamma\times\Gamma,1)=K(\Gamma,1)\times K(\Gamma,1)$ up to homotopy. But $K(d,1)$ is homotopic to the diagonal map $\triangle\co K(\Gamma,1)\to K(\Gamma,1)\times K(\Gamma,1)$. Hence $\TC_{K(\Gamma,1)}(K(\mathcal{Z},1))\leq 1$, and the result follows on applying Theorem \ref{fibration}.
\end{proof}

Let $\mathscr{G}$ denote the class of finitely generated torsion-free nilpotent groups. Any $\Gamma\in \mathscr{G}$ admits a central series
\[
\Gamma = \Gamma_0 \geq \Gamma_1 \geq \cdots \geq \Gamma_n = \{1\}
\]
with all quotients $\Gamma_i/\Gamma_{i+1}$ free abelian. The sum of the ranks of these quotients is the {\em rank}, (or {\em Hirsch number}) of $\Gamma$, and is denoted $\rank(\Gamma)$. It is a well-known fact (see Gruenberg \cite{Gru}, Section 8.8) that $\rank(\Gamma) = \cd(\Gamma)$ for $\Gamma\in\mathscr{G}$, where $\cd(\Gamma)$ is the cohomological dimension (the largest integer $k$ such that the group cohomology $H^k(\Gamma;A)\neq 0$ for some $\Gamma$-module $A$). On the other hand, Eilenberg and Ganea \cite{EG57} and Stallings \cite{Sta} have shown that $\cd(\Gamma)+1=\cat(\Gamma):=\cat(K(\Gamma,1))$ for any finitely generated group. Therefore, for $\Gamma\in\mathscr{G}$ we have $\cat(\Gamma)=\rank(\Gamma)+1$.

\begin{cor}\label{nil} Let $\Gamma$ be a finitely generated torsion-free nilpotent group with centre $\mathcal{Z}\leq \Gamma$. Then
\begin{equation}\label{nilineq}
\TC(\Gamma)\leq 2\cdot\rank(\Gamma)-\rank(\mathcal{Z})+1.
\end{equation}
\end{cor}
\begin{proof} Since the class $\mathscr{G}$ is closed under formation of finite direct products and subgroups, we have $\mathcal{Z}, \Gamma\times \Gamma\in\mathscr{G}$. The quotient $H:=(\Gamma\times\Gamma)/\mathcal{Z}$ is finitely generated and nilpotent. It is easily seen that torsion in $H$ would imply torsion in $\Gamma/\mathcal{Z}$. However the latter group is torsion-free, as follows from \cite[1.2.20]{LR}.  Therefore all the groups in the extension (\ref{extension}) above are in the class $\mathscr{G}$. Since the rank is additive on extensions in this class, we have
\[
\TC(\Gamma)\leq \cat(H)=\rank(H)+1 = \rank(\Gamma\times\Gamma)-\rank(\mathcal{Z})+1=2\cdot\rank(\Gamma)-\rank(\mathcal{Z})+1
\]
as required.
\end{proof}
\begin{rem}{\em It is natural to ask to what extent the inequality (\ref{nilineq}) is sharp. To this end, one seeks lower bounds for $\TC(X)$ when $X=K(\Gamma,1)$ is a nilmanifold, in terms of cohomology theory. If $\Gamma$ is abelian, then $X$ is a torus and (\ref{nilineq}) is an equality in this case. In all other cases, the presence of non-trivial Massey products in rational cohomology suggest that the methods of \cite{G} may be of use. We hope to return to this question in a future paper, noting here that $\TC(X)\geq\cat(X)=\dim(X)+1=\rank(\Gamma)+1$ is currently the best known lower bound.
}\end{rem}
\section{Closed maps}

In this section we extend Theorem \ref{fibration} to closed maps $q\co X\times X\to Y$. (Recall that a map is {\em closed} if it sends closed sets to closed sets; in particular, any map $q\co X\times X\to Y$ is closed if $X\times X$ is compact and $Y$ is Hausdorff.) This greater generality comes at the expense of replacing $\cat(Y)$ with the potentially larger number $\dim(Y)+1$. The results in this section were inspired by the corresponding results for category appearing in the paper \cite{OW}.

\begin{lemma}[{\cite{OW},\cite[Lemma A.4]{CLOT}}]\label{milnor} Let $B$ be a paracompact space with covering dimension $\dim(B)=n$, and let $\mathcal{U}=\lbrace U_\alpha\rbrace$ be any open cover of $B$. Then there exists an open refinement $\mathcal{G}=\lbrace G_{i\beta}\rbrace, i=1,\ldots ,n+1$ of $\mathcal{U}$ such that $G_{i\beta}\cap G_{i\beta'}=\emptyset$ for $\beta\neq \beta'$.
\end{lemma}

\begin{lemma}[{\cite{OW},\cite[Lemma 9.39]{CLOT}}]\label{saturate} Let $q\co B\to Y$ be closed, $y\in Y$. If $U\subseteq B$ is an open set with $q^{-1}(y)\subseteq U$, then there exists a saturated open set $V$ such that $q^{-1}(y)\subseteq V \subseteq U$ (recall that $V\subseteq B$ is {\em saturated} if it is the inverse image of an open set in $Y$).
\end{lemma}

\begin{thm}\label{main} Let $X$ be a normal ENR, and let $q\co X\times X\to Y$ be a closed map with $Y$ paracompact. Suppose further that $\TC_X(q^{-1}(y))\leq n$ for each $y\in Y$. Then
\[
\TC(X)\leq(\dim Y+1)\cdot n.
\]
\end{thm}
 \begin{proof} Let $O_{y}$ denote the fibre $q^{-1}(y)$. By assumption, $\TC_X(O_{y})\leq n$. So by Lemma \ref{open} we can cover $O_{y}$ by sets $U_1^{y},\ldots ,U_n^{y}$ open in $X\times X$ admitting local sections of $\pi\co \mathscr{P}X\to X\times X$. By Lemma \ref{saturate} there exists for each fibre a saturated open set $V^{y}$ such that $$O_{y}\subseteq V^{y}\subseteq U^{y}=\bigcup_i U_i^{y}.$$
Now $V^{y}=q^{-1}\tilde{V}^{y}$ by saturation. Then $\lbrace\tilde{V}^{y}\rbrace_{y\in Y}$ is an open cover of $Y$. Let $k=\dim(Y)$. By Lemma \ref{milnor} there exists a refinement $\lbrace \tilde{G}_{i\beta}\rbrace, i=1,\ldots ,k+1$ such that each $\tilde{G}_i$ is a disjoint union of open sets $\cup_\beta\tilde{G}_{i\beta}$, each of which is contained in some $\tilde{V}^{y}$.

Let $G_i=q^{-1}(\tilde{G}_i)$ and $G_{i\beta}=q^{-1}(\tilde{G}_{i\beta})$. Suppose $\tilde{G}_{i\beta}\subseteq \tilde{V}^{y}$. Then
$$G_{i\beta}=q^{-1}(\tilde{G}_{i\beta})\subseteq q^{-1}\tilde{V}^{y}\subseteq U^{y}.$$

Now define $G_{i\beta j}=G_{i\beta}\cap U^{y}_j$ for each $j=1,\ldots ,n$, and set $G_{ij}=\bigcup_\beta G_{i\beta j}$. Note that $G_{ij}$ is a disjoint union of open sets admitting local sections of $\pi$. Note also that the $G_{ij}$ cover $X\times X$, and there are $(k+1)\cdot n$ of them. Hence $\TC(X\times X)\leq (k+1)\cdot n$ as required.
\end{proof}

\begin{cor} With $X$, $Y$ and $q\co X\times X\to Y$ as in Theorem \ref{main}, we have
\[
\TC(X)\leq (\dim(Y)+1)(\dim(q)+1),
\]
where $\dim(q)=\sup\{\dim q^{-1}(y)\mid y\in Y\}$ denotes the covering dimension of the map $q$.
\end{cor}
\begin{proof} This follows immediately from Theorem \ref{main} and inequality (\ref{dim}).
\end{proof}

In the next section we will apply Theorem \ref{main} to estimate the topological complexity of spaces with group actions. It may also be applied to give a quick proof of (part of) the result of Farber, Tabachnikov and Yuzvinsky \cite{FTY}, relating the topological complexity of real projective spaces to their immersion dimension.

For any natural number $n$, let $P^n$ denote real projective $n$-space, and let $w\in H^1(P^n;\Z_2)$ denote the generator. Recall that a map $a\co P^n\times P^n\to P^r$ is called {\em axial} if the restriction of $a$ to each factor $\{*\}\times P^n$ and $P^n\times\{*\}$ is homotopic to the inclusion $P^n\hookrightarrow P^r$. An equivalent condition is that $a^*(w)=1\times w + w\times 1\in H^1(P^n\times P^n;\Z_2)$. The main Theorem 6.1 of \cite{FTY} can be reformulated as
\begin{center}
{\em $\TC(P^n)\leq r+1$ if and only if there exists an axial map $a\co P^n\times P^n\to P^r$.}
\end{center}
It follows that $\TC(P^n)$ equals one plus the immersion dimension for $n\neq 1,3, 7$.

Here we prove the if part of their statement using Theorem \ref{main}.
\begin{cor}\label{axial} If there exists an axial map $a\co P^n\times P^n\to P^r$, then $\TC(P^n)\leq r+1$.
\end{cor}
\begin{proof} The conclusion is always true when $r\geq 2n$, by the general dimensional upper bound $\TC(P^n)\leq 2n+1$. If $r\leq n$ then $r=n$ and we must be in one of the exceptional cases $n=1,3,7$, in which case $\TC(P^n)=n+1$. Hence we may assume that $n<r<2n$.

Since the axial condition is homotopical, we may assume that $a\co P^n\times P^n\to P^r$ is a smooth proper submersion. Note that the map $a$ satisfies all the hypotheses of Theorem \ref{main}, so it suffices to check that $\TC_{P^n}(a^{-1}(y))\leq 1$ for every $y\in P^r$.

Each fiber $F_y:=a^{-1}(y)\subseteq P^n\times P^n$ is a smooth submanifold of dimension $2n-r$.  Let $\iota\co F_y\hookrightarrow P^n\times P^n$ be the inclusion, and let $p_1,p_2\co P^n\times P^n\to P^n$ be the projections. Recall that $\TC_{P^n}(F_y)\leq 1$ if and only if $p_1\circ\iota\simeq p_2\circ\iota\co F_y\to P^n$. Since $F_y$ has the homotopy type of a CW-complex of dimension $2n-r<n$, there are isomorphisms $[F_y,P^n]\cong[F_y,P^\infty]\cong H^1(F_y;\Z_2)$, so we only need to show that $(p_1\circ\iota)^*(w)= (p_2\circ\iota)^*(w)$, where $w\in H^1(P^n;\Z_2)$ is the generator. But
\begin{align*}
(p_1\circ\iota)^*(w)= (p_2\circ\iota)^*(w) & \iff \iota^*(w\times 1)=\iota^*(1\times w) \\
 & \iff \iota^*(1\times w + w\times 1)=0 \\
 &\iff \iota^*a^*(w) = 0,
 \end{align*}
 where the latter equality is clearly true since $a\circ\iota$ is constant.
 \end{proof}

\section{Group actions}

We now apply the results of previous sections to obtain upper bounds for the topological complexity in the presence of group actions. In this section, $G$ will always be a compact Lie group, and $X$ a closed, connected smooth manifold. We quote several results from the theory of compact transformation groups, all of which may be found in the books of Bredon \cite{Bre} or tom Dieck \cite{tD}.

We fix some notation. For a (left) $G$-action $\rho\co G\times X\to X$ we will use the shorthand $\rho(g,x) = g\cdot x$.  The {\em orbit} of the element $x\in X$ under this action is $G(x)=\lbrace g\cdot x \mid g\in G\rbrace\subseteq X$. The {\em stabiliser} of $x\in X$ is the subgroup $G_x = \lbrace g\in G\mid g\cdot x = x\rbrace$ in $G$. The action is {\em free} if each stabiliser $G_x$ is the trivial subgroup $\{e\}$ (where $e\in G$ is the identity element). The action is {\em semi-free} if each stabiliser $G_x$ is either $\{e\}$ or $G$.

 The {\em fixed point set} of a subgroup $H$ of $G$ is the subspace of $X$ defined by $F(H,X)=\lbrace x\in X\mid h\cdot x = x\mbox{ for all }h\in H\rbrace$.

Let $X/G$ denote the space of orbits of the action, given the quotient topology via the {\em orbit map} $p\co X\to X/G$ which sends $x\in X$ to  $G(x)$. Since $X$ is compact, so is $X/G$. Since $G$ is compact and Hausdorff, $X/G$ is Hausdorff. Note that $p$ is therefore closed. If the action is free, then $p$ is a principal $G$-bundle, hence a fibration.

 {\em Evaluation} at a point $x\in X$ defines a map $\ev_x\co G\to X$, given by $\ev_x(g)=g\cdot x$, whose image is $G(x)$. Since $G$ is compact, the induced map $q_x\co G/G_x\to G(x)$ on cosets given by $q_x(gG_x)=g\cdot x$ is a homeomorphism onto the orbit. The orbit $G(x)$ is said to be of {\em type} $G/G_x$.

A $G$-action $G\times X\to X$ is called {\em locally smooth} if there is a linear tube $\varphi\co G\times_H V\to X$ about every orbit of type $G/H$ (here $V$ is an orthogonal representation of $H$). We refer to Bredon \cite{Bre}, Chapter IV for background material on this concept. Note that smooth actions are locally smooth. Locally smooth actions have {\em principal orbits}, that is orbits of type $G/H$ where $H$ is conjugate to a subgroup of any stabiliser group $G_x\subseteq G$. Principal orbits are orbits of maximal dimension. An orbit of maximal dimension which is not principal is called an {\em exceptional orbit}. Our primary interest in locally smooth actions stems from the fact that the covering dimension of their orbit spaces is well understood. In fact, let $P$ be a principal orbit of a locally smooth action $G\times X\to X$. Then \cite[Theorem IV.3.8]{Bre} states that
\begin{equation}\label{orbitdim}
\dim(X/G) = \dim(X)-\dim(P).
 \end{equation}

The above comments and notations apply also to actions of $G$ on the product $X\times X$. Applying Theorems \ref{fibration} and \ref{main} to the orbit map $p\co X\times X\to (X\times X)/G$, we obtain the following.
 \begin{cor}\label{group}
 Suppose that $G$ acts locally smoothly on $X\times X$, and each orbit $G(x,y)$ has $\TC_X(G(x,y))\leq n$. If $P\subseteq X\times X$ is a principal orbit, then
\[
 \TC(X)\leq (2\dim(X)-\dim(P)+1)\, n.
\]
 If the action is free, then
 \[
 \TC(X)\leq\cat((X\times X)/G)\, n.
 \]
 \end{cor}
Obvious candidates for $G$-actions on the product $X\times X$ are diagonal actions. Recall that given an action of $G$ on $X$ there is an associated {\em diagonal action}
\[
G\times X\times X\to X\times X, \qquad \big(g,(x,y)\big)\mapsto (g\cdot x,g\cdot y)
\]
of $G$ on $X\times X$. The diagonal action associated to a locally smooth $G$-action is locally smooth. Notice that the stabiliser of a point $(x,y)$ is $G_{(x,y)} = G_x\cap G_y$, the intersection of the stabilisers of $x$ and $y$. It follows that a principal orbit $P'$ of the diagonal action has $\dim(P')\geq\dim(P)$. The following result gives sufficient conditions for the orbits $G(x,y)\subseteq X\times X$ of a diagonal action to have $\TC_X(G(x,y))\leq 1$.

\begin{thm}\label{path} Suppose $G$ acts locally smoothly on $X$ with principal orbit $P$. Suppose further that for any $(x,y)\in X\times X$ one of the following conditions holds:
\begin{enumerate}
\item $F(G_x\cap G_y,X)$ is path-connected;
\item $x,y\in F(G,X)$.
\end{enumerate}
Then each orbit of the diagonal action of $G$ on $X\times X$ has $\TC_X(G(x,y))\leq 1$, and consequently
\[
\TC(X)\leq 2\dim (X) - \dim(P) + 1.
\]
\end{thm}

\begin{proof} Fix an orbit $G(x,y)$, and choose a representative pair $(x,y)\in X\times X$. Note that $x\in F(G_x,X)\subseteq F(G_x\cap G_y,X)$ and $y\in F(G_y,X)\subseteq F(G_x\cap G_y,X)$. In case (1), we may choose a path $\gamma\co I\to F(G_x\cap G_y,X)$ with $\gamma(0)=x$ and $\gamma(1)=y$. Then evaluation at $\gamma$ induces the map $q_\gamma$ in the commutative diagram
\[
\xymatrix{
G \ar[rd] \ar[rr]^{\mathrm{ev}_\gamma} & & X^I, & g  \ar[rr] \ar[rd] &  & g\gamma \\
 & G/(G_x\cap G_y) \ar[ur]_{q_\gamma} & & & g(G_x\cap G_y) \ar[ur] &
 }
 \]
Now $q_{(x,y)}\co G/(G_x\cap G_y)\to G(x,y)$ is a homeomorphism onto the orbit. Hence we may define a local section $s\co G(x,y)\to X^I$ of $\pi$ by setting $s=q_\gamma\circ q_{(x,y)}^{-1}$, $s(gx,gy) = g\gamma$.

In case (2) the orbit $G(x,y)=\{(x,y)\}$ is a point, and so $\TC_X(G(x,y))=1$ trivially.
 \end{proof}

\begin{cor}\label{semi} If $G$ acts locally smoothly, non-trivially and semi-freely on $X$, then
 \[
 \TC(X)\leq 2\dim(X)-\dim(G)+1.
 \]
 If $G$ acts locally smoothly and freely on $X$, then
 \[
 \TC(X)\leq \cat((X\times X)/G)\leq 2\dim(X)-\dim(G)+1.
 \]
 \end{cor}
 \begin{proof}
 It is easy to see that the conditions of Theorem \ref{path} are satisfied, and the dimension of a principal orbit is $\dim(P)=\dim(G)$.
\end{proof}

\begin{cor}[Homology spheres]
Let $\Sigma$ be an odd dimensional integral homology sphere. Suppose $\Sigma$ admits a non-trivial locally smooth action of $S^1$ for which the exceptional orbits are all of type $\Z_p$, $p$ prime. Then
\[
\TC(\Sigma)\leq 2\dim(\Sigma).
\]
\end{cor}

\begin{proof}
The non-triviality of the action implies that a principal orbit has dimension one. The intersections of stabilizers are either trivial, all of $S^1$, or $\Z_p$ for some prime $p$. Hence it is enough by Theorem \ref{path} to show that the fixed point set $F(\Z_p,\Sigma)$ of each subgroup $\Z_p\subset S^1$ is either empty or path-connected.

We use results from Smith theory. When $p$ is odd, since $\Sigma^n$ is a $\Z_p$-cohomology $n$-sphere it follows that $F(\Z_p,\Sigma)$ is a $\Z_p$-cohomology $r$-sphere, where $n-r$ is even, and hence $r$ is odd (see for example \cite[Theorem III.7.1]{Bre}). The case $r=-1$ is included and corresponds to $F(\Z_p,\Sigma)=\emptyset$. Hence $F(\Z_p,\Sigma)$ is either empty or path-connected.

When $p=2$ the same result implies that $F=F(\Z_2,\Sigma)$ is a $\Z_2$-cohomology sphere, but possibly zero dimensional. To show that this cannot occur, we invoke the extension of Smith's result due to Heller \cite{He} and Swan \cite{Swa} (see also \cite[Theorem VII.2.2]{Bre}). Note that $\Z_2\subset S^1$ acts trivially on the integral cohomology $H^*(\Sigma;\Z)$, since $S^1$ is path-connected. It follows that
\[
\mathrm{rk}\, H^0(F;\Z_2)\leq \sum_{i\geq 0} \mathrm{rk}\,H^{2i}(F;\Z_2)\leq\sum_{i\geq 0} \mathrm{rk}\,H^{2i}(\Sigma;\Z_2) = 1,
\]
the last equality since $\Sigma$ is an odd-dimensional $\Z_2$-cohomology sphere.
\end{proof}

\begin{exam}{\em Let $p,q,r>1$ be distinct primes. The {\em Brieskorn variety}
\[
\Sigma(p,q,r) = \lbrace (z_1,z_2,z_3)\in \C^3 \mid z_1^p + z_2^q + z_3^r = 0\rbrace\cap S^5\subset\C^3
\]
is an integral homology $3$-sphere.  Smooth circle actions on the Brieskorn varieties $\Sigma(p,q,r)$ have been studied by Orlik. In particular, section 9 of \cite{Or} describes a fixed-point free action with exceptional orbits of type $\Z_p$, $\Z_q$ and $\Z_r$. Hence
\[
\TC(\Sigma(p,q,r))\leq 6.
\]
Note that the $\Sigma(p,q,r)$ are  not simply-connected in general. The author does not know how to obtain this upper bound by other methods.
}\end{exam}
\begin{rem}{\em The above example includes the case of the Poincar\'{e} sphere $\Sigma=\Sigma(2,3,5)$, which admits an alternative description as the homogeneous manifold $SO(3)/I$, where $I$ denotes the icosahedral group of order $60$. Thus $\Sigma$ admits a natural action of the $3$-dimensional group $SO(3)$ with finite stabilisers, hinting that the above result may be suboptimal in this case. However, we were unable to show that $\TC(\Sigma)\leq 4$ by our methods, due to the fact that dihedral groups can seemingly occur as intersections of stabilisers, giving disconnected fixed point sets $F(G_x\cap G_y,\Sigma)=S^0$. Note that $\cat(\Sigma)=4$, see \cite{GG}.}
\end{rem}

\begin{cor} Let $\phi\co X\to X$ be a diffeomorphism of prime period $p$ (meaning $\phi^p = \mathrm{Id}_X$). Let $M=M_\phi = I \times X / (1,x)\sim(0,\phi(x))$ be the mapping torus of $\phi$. If the fixed-point set $X^\phi$ is path-connected, then
\[
\TC(M)\leq 2\dim M.
\]
\end{cor}
\begin{proof} Under the natural smooth action of $S^1$ on $M$, the stabiliser of a point $[t,x]$ is either $\{1\}$ or $\Z_p$, depending on whether $x\in X$ is free or fixed under $\phi$. Note that the fixed point set $F(\Z_p,M)$ is diffeomorphic with $S^1\times X^\phi$, hence is path-connected. Now apply Theorem \ref{path}.
\end{proof}

\begin{exam}{\em Consider the involution $\phi\co S^n\to S^n$ given by reflection in the equator $S^{n-1}\subset S^n$. For $n>1$ the fixed-point set of $\phi$ is path-connected, and so $\TC(M_\phi)\leq 2n+2$. For example, the $3$-dimensional Klein bottle $K^3=M_{\phi\co S^2\to S^2}$ has $\TC(K^3)\leq 6$. However, our methods do not apply to the $2$-dimensional Klein bottle $K^2$, which can be viewed as the mapping torus of the complex conjugation $\phi\co S^1\to S^1$ with fixed-point set $S^0$.
}\end{exam}

\end{document}